\documentclass[10pt]{amsart}
\usepackage{amsmath,amsthm,fixltx2e}
\usepackage{color}
\usepackage{amsxtra}
\usepackage{amscd}
\usepackage{amsfonts}
\usepackage{amssymb}
\usepackage{eucal}

\usepackage[all]{xy}

\newtheorem{theorem}{Theorem}[section]
\newtheorem{cor}[theorem]{Corollary}
\newtheorem{lem}[theorem]{Lemma}
\newtheorem{prop}[theorem]{Proposition}

\newtheorem{thm}[theorem]{Theorem}
\newtheorem{claim}[theorem]{Claim}

\theoremstyle{definition}
\newtheorem{defn}{Definition}[section]

\theoremstyle{remark}
\newtheorem{rem}{Remark}[section]

\newcommand{\nc}{\newcommand}

\nc\ol{\overline} \nc\ul{\underline} \nc\wt{\widetilde}
\nc{\z}{\zeta}

\nc{\ZZ}{{\mathbb Z}} \nc{\NN}{{\mathbb N}} \nc{\CC}{{\mathbb C}}
\nc{\RR}{{\mathbb R}}

\nc{\A}{{\mathbb A}}  \nc\U{U}

\nc{\F}{{\mathcal F}} \nc{\N}{{\mathcal N}} \nc{\Aa}{{\mathcal A}}
\nc{\Oo}{{\mathcal O}} \nc{\Hh}{{\mathcal H}}

\DeclareMathOperator{\gr}{\mathrm{gr}}
\DeclareMathOperator{\Gr}{\mathrm{Gr}}

\DeclareMathOperator{\diag}{\mathrm{diag}}
\DeclareMathOperator{\tr}{\mathrm{tr}}

\DeclareMathOperator{\Ker}{\mathrm{Ker}}

\DeclareMathOperator{\cl}{\mathrm{cl}}
\DeclareMathOperator{\s}{\mathrm{S}}

\DeclareMathOperator{\Res}{\mathrm{Res}}
\DeclareMathOperator{\Rad}{\mathrm{Rad}}
\DeclareMathOperator{\Pf}{\mathrm{Pf}}

\newcommand\q{\mathfrak q}
\newcommand\h{\mathfrak h}
\newcommand\C{\mathfrak c}
\newcommand\aaa{\mathfrak a}
\newcommand\m{\mathfrak m}

\newcommand{\zz}{\mathfrak{z}}

\newcommand{\gl}{\mathfrak{gl}}
\newcommand{\ssl}{\mathfrak{sl}}
\newcommand{\g}{\mathfrak{g}}
\newcommand{\spn}{\mathfrak{sp}}
\newcommand{\so}{\mathfrak{so}}

\newcommand{\ad}{\mathop{\rm ad}\nolimits}
\nc{\iso}{{\stackrel{\sim}{\longrightarrow}}}

\begin{document}

\author[Alexander Tsymbaliuk]{Alexander Tsymbaliuk}
 \address{Independent University of Moscow, 11 Bol'shoy Vlas'evskiy per., Moscow 119002, Russia}
 \curraddr{Department of Mathematics, MIT, 77 Massachusetts Avenue, Cambridge, MA  02139, USA}
 \email{sasha\_ts@mit.edu}

\title[Infinitesimal Hecke algebras of $\so_N$]
 {Infinitesimal Hecke algebras of $\so_N$}

\begin{abstract}
  In this article we classify all infinitesimal Hecke algebras of $\g=\so_N$.
 We establish isomorphism of their universal versions and the $W$-algebras of $\so_{N+2m+1}$ with
 a $1$-block nilpotent element of the Jordan type $(1,\ldots,1,2m+1)$.
  This should be considered as a continuation of~\cite{LT}, where the analogous results were obtained for the cases of $\g=\gl_n,\spn_{2n}$.
\end{abstract}

\maketitle

\section*{Introduction}
  In this paper we consider infinitesimal Hecke algebras of
  $\so_N$.\footnote{\ We assume that $N\geq 3$.}
 Although their theory runs along similar lines as for the cases of
 $\gl_N$ and $\spn_{2N}$, they have not been investigated before.

  We obtain the classification result in Theorem~\ref{main 2} (compare to~\cite[Theorem 4.2]{EGG}),
 compute the Poisson center of the corresponding Poisson algebras in Theorem~\ref{deformed case} (compare to~\cite[Theorems 5.1 and 7.1]{DT},
 compute the first non-trivial central element in Theorem~\ref{main 5} (compare to~\cite[Theorem 3.1]{DT})
 and derive the isomorphism with the corresponding $W$-algebras in Theorems~\ref{main 3},~\ref{main 4} (compare to~\cite[Theorems 7 and 10]{LT}).

  Together with~\cite{LT}, this covers all basic cases of the infinitesimal Hecke algebras on the one side and
 the classical $W$-algebras with a $1$-block nilpotent element, on the other.
  However, we would like to emphasize that the theory of infinitesimal/continuous Hecke algebras is much more complicated in general and has not been developed yet.

\medskip
 This paper is organized as follows:

\medskip
 $\bullet$
  In Section 1, we recall the definitions of the continuous and infinitesimal Hecke algebras of type $(G,V)$ (respectively $(\g,V)$).
 We formulate Theorems~\ref{main 1} and~\ref{main 2}, which classify all such algebras
 for the cases of $(\mathrm{SO}_N,V_N)$ and $(\so_N,V_N)$, respectively.

 We also recall the definitions and basic results about the finite $W$-algebras.

\medskip
 $\bullet$
  In Section 2, we prove Theorem~\ref{main 1}.

\medskip
 $\bullet$
  In Section 3, we prove Theorem~\ref{main 2} by computing explicitly the corresponding integral.

\medskip
 $\bullet$
  In Section 4, we compute the Poisson center of the classical analogue $H^{\cl}_\z(\so_N,V_N)$.

\medskip
 $\bullet$
  In Section 5, we introduce the universal length $m$ infinitesimal Hecke algebra $H_m(\so_N,V_N)$.
 In Theorem~\ref{main 3} (and its Poisson counterpart Theorem~\ref{main 4}) we establish an abstract isomorphism between
 the algebras $H_m(\so_N,V_N)$ and the $W$-algebras $\U(\so_{N+2m+1},e_m)$.

\medskip
 $\bullet$
  In Section 6, we find a non-trivial central element of $H_\z(\so_N,V_N)$, called the \emph{Casimir element} of $H_\z(\so_N,V_N)$.
  This can be used to establish the isomorphism of Theorem~\ref{main 3} explicitly.

\subsection*{Acknowledgments}

 The author is grateful to P.~Etingof and I.~Losev for numerous stimulating discussions.
 Special thanks are due to S.~Sam for pointing out the result of Claim~\ref{radical}.
 Finally, the author is grateful to F.~Ding for his comments on the first version of the
 paper and to the referee for the useful comments on the final version of this paper.


\medskip

\newpage

\section{Basic definitions}

\subsection{Algebraic distributions}
$\ $

  For an affine scheme $X$ of finite type over $\CC$, let $\Oo(X)$ be the
 algebra of regular functions on $X$ and $\Oo(X)^*$ be the dual
 space, called the space of \emph{algebraic distributions}.
  Note that $\Oo(X)^*$ is a module over $\Oo(X)$: for $f\in \Oo(X),\ \mu\in \Oo(X)^*$ we can define $f\cdot\mu$ by
 $\langle f\cdot\mu,g\rangle= \langle \mu,fg\rangle$ for all $g\in \Oo(X)$.
  For a closed subscheme $Z\subset X$, we say that an algebraic
 distribution $\mu$ on $X$ is supported on the scheme $Z$ if $\mu$
 annihilates the defining ideal $I(Z)$ of $Z$. If $Z$ is reduced, we
 say that $\mu\in \Oo(X)^*$ is \emph{set-theoretically supported} on
 the set $Z$ if $\mu$ annihilates some power of $I(Z)$.

  Let $G$ be a reductive algebraic group and $\rho:G\to \mathrm{GL}(V)$ be a finite dimensional algebraic representation of $G$.
 First note that $\Oo(G)^*$ is an algebra with respect to the convolution. Moreover, $\delta_{1_G}$ is the unit of this algebra.
 Next, we consider the semi-direct product $\Oo(G)^*\ltimes TV$, that is,
 the algebra generated by $\mu\in \Oo(G)^*$ and $x\in V$ with the relations
  $$x\cdot \mu=\sum_i (v_i^*,gx)\mu \cdot v_i\ \ \ \mathrm{for\ all}\ x\in V,\ \mu\in \Oo(G)^*,$$
 where $\{v_i\}$ is a basis of $V$ and $\{v_i^*\}$ the dual basis of $V^*$, while
 $(v_i^*,gx)\mu$ denotes the product of the regular function
 $(v_i^*,gx)$ and the distribution $\mu$.

  We will denote the vector space of length $N$ columns by $V_N$, so that
 there are natural actions of $\mathrm{GL}_N, \mathrm{Sp}_N, \mathrm{SO}_N$ on $V_N$.
  Let us also denote the action of $g\in G$ on $x\in V$ by $x^g$.

\subsection{Continuous Hecke algebras}
$\ $

  We recall the definition of the continuous Hecke algebras of $(G,V)$ following~\cite{EGG}.

 Given a reductive algebraic group $G$, its finite dimensional algebraic representation $V$ and a skew-symmetric
 $G$-equivariant $\CC$-linear map $\kappa:V\times V\to \Oo(G)^*$, we set
    $$\mathcal{H}_\kappa(G,V):=\Oo(G)^*\ltimes TV/([x,y]-\kappa(x,y)|\ x,y\in V).$$

 Consider an algebra filtration on $\mathcal{H}_\kappa(G,V)$ by setting $\deg(V)=1$ and $\deg (\Oo(G)^*)=0$.

\begin{defn}\label{continuous}\cite{EGG}
  We say that $\mathcal{H}_\kappa(G,V)$ \emph{satisfies the PBW property} if the natural surjective map
 $\Oo(G)^*\ltimes SV \twoheadrightarrow \mathrm{gr}\ \mathcal{H}_\kappa(G,V)$ is an isomorphism, where $SV$ denotes the symmetric algebra of $V$.
 We call these $\mathcal{H}_\kappa(G,V)$ the \emph{continuous Hecke algebras} of $(G,V)$.
\end{defn}

  According to~\cite[Theorem 2.4]{EGG}, $\mathcal{H}_\kappa(G,V)$ satisfies the PBW property if and only if $\kappa$ satisfies the \emph{Jacobi identity}:
\begin{equation}\label{Jacobi}\tag{\dag}
 (z-z^g)\kappa(x,y)+(y-y^g)\kappa(z,x)+(x-x^g)\kappa(y,z)=0 \ \ \mathrm{for\ all}\ x,y,z\in V.
\end{equation}

  Define the closed subscheme $\Phi\subset G$ by the equation $\wedge^3(1-g_{\mid_V})=0$.
 The set of closed points of $\Phi$ is the set $S=\{g\in G:\mathrm{rk}(1-g_{\mid_V})\leq 2\}$. We have:

\begin{prop}\label{0}~\cite[Proposition 2.8]{EGG}
  If the PBW property holds for $\mathcal{H}_\kappa(G,V)$, then
 $\kappa(x,y)$ is supported on the scheme $\Phi$ for all $x,y\in V$.
\end{prop}

 The classification of all $\kappa$ satisfying (\dag) was obtained in~\cite{EGG} for the following two cases:

\noindent
 $\bullet$
  for the pairs $(G,\mathfrak{h}\oplus\mathfrak{h}^*)$ with $\mathfrak{h}$ being an irreducible faithful $G$-representation of real or complex type
  (see~\cite[Theorem 3.5]{EGG}),

\noindent
 $\bullet$
  for the pair $(\mathrm{Sp}_{2n},V_{2n})$ (see~\cite[Theorem 3.14]{EGG}).

  For general continuous Hecke algebras such a classification is not known at the moment.
 However, a particular family of those was established in~\cite[Theorem 2.13]{EGG}:

\begin{prop}\label{1}
  For any $\tau\in (\Oo(\Ker \rho)^*\otimes \wedge^2 V^*)^G$ and $\upsilon\in (\Oo(\Phi)^*\otimes \wedge^2 V^*)^G$,
 the pairing $\kappa_{\tau,\upsilon}(x,y):=\tau(x,y)+\upsilon((1-g)x,(1-g)y)$ satisfies the Jacobi identity.
\end{prop}

  Our first result is a full classification of all $\kappa$ satisfying (\dag) for the case of $(\mathrm{SO}_N,V_N)$,
 which is similar to the aforementioned classification for $(\mathrm{Sp}_{2n},V_{2n})$.
 But it turns out that $\Phi$ is not reduced in this case and so we need a more detailed argument.

\begin{thm}\label{main 1}
  The PBW property holds for $\Hh_\kappa(\mathrm{SO}_N,V_N)$ if and only if there exists an $\mathrm{SO}_N$-invariant distribution
 $c\in \Oo(S)^*$ such that $\kappa(x,y)=((g-g^{-1})x,y)c$ for all $x,y\in V_N$.
\end{thm}

  The proof of this theorem is presented in Section 2.

\subsection{Infinitesimal Hecke algebras}
$\ $

  For any triple $(\g,V,\kappa)$ of a Lie algebra $\g$, its representation $V$ and a $\g$-equivariant $\CC$-bilinear pairing
 $\kappa: \wedge^2 V\to U(\g)$, we define
   $$H_\kappa(\g,V):=U(\g)\ltimes TV/([x,y]-\kappa(x,y)|\ x,y\in V).$$

 Endow this algebra with a filtration by setting $\deg(V)=1,\ \deg(\g)=0$.

\begin{defn}\label{infinitesimal}\cite[Section 4]{EGG}
  We call this algebra the \emph{infinitesimal Hecke algebra} of $(\g,V)$ if it satisfies the \emph{PBW property},
 that is, the natural surjective map $U(\g)\ltimes SV\twoheadrightarrow \gr H_{\kappa}(\g,V)$ is an isomorphism.
\end{defn}

  Any such algebra gives rise to a continuous Hecke algebra
    $$\Hh_{\kappa}(G,V):=\Oo(G)^*\otimes_{U(\g)}H_{\kappa}(\g,V),$$
 where $U(\g)$ is identified with the subalgebra $\Oo(G)_{1_G}^*\subset\Oo(G)^*$,
 consisting of all algebraic distributions set-theoretically supported at $1_G\in G$.

  In particular, having a full classification of the continuous Hecke algebras of type $(G,V)$
 yields a corresponding classification for the infinitesimal Hecke algebras of $(\mathrm{Lie}(G),V)$.
  The latter classification was determined explicitly for the cases of $(\g,V)=(\gl_n,V_n\oplus
 V_n^*),(\spn_{2n},V_{2n})$ in~\cite[Theorem 4.2]{EGG}.

\medskip
 To formulate our classification of infinitesimal Hecke algebras $H_\kappa(\so_N,V_N)$ we define:

\medskip
\noindent
 $\bullet$ $\gamma_{2j+1}(x,y)\in S(\so_N)\simeq \CC[\so_N]$
 by
\begin{equation*}
   (x,A(1+\tau^2 A^2)^{-1}y)\det(1+\tau^2A^2)^{-1/2}=\sum_{j\geq 0}
   \gamma_{2j+1}(x,y)(A)\tau^{2j},\ \ A\in \so_N,
\end{equation*}
 where we formally set $(1+T)^\alpha:=1+\sum_{n=1}^\infty \frac{\alpha(\alpha-1)\cdots(\alpha-n+1)}{n!}T^n$
 for $\alpha\in \RR,\ T\in \tau^2\CC[\tau^2]$;

\noindent
 $\bullet$ $r_{2j+1}(x,y)\in U(\so_N)$ to be the symmetrization of $\gamma_{2j+1}(x,y)\in S(\so_N)$.

\medskip
 The following theorem is proved in Section 3:
\begin{thm}\label{main 2}
  The PBW property holds for $H_\kappa(\so_N,V_N)$ if and only if $\kappa=\sum_{j=0}^k \z_jr_{2j+1}$
 for some non-negative integer $k$ and parameters $\z_0,\ldots,\z_k\in \CC$.
\end{thm}

 This theorem is very similar to the analogous results for the pairs $(\gl_n,V_n\oplus V_n^*)$ and $(\spn_{2n},V_{2n})$.
 We denote the corresponding algebra by $H_\z(\so_N,V_N)$ for $\kappa$ of the above form.

\begin{rem}
 (a) For $\z_0\ne 0$, we have $H_{\z_0r_1}(\so_N,V_N)\simeq U(\so_{N+1})$.
 Thus, for an arbitrary $\z$ we can regard $H_\z(\so_N,V_N)$ as a deformation of $U(\so_{N+1})$.

\noindent
 (b) Theorem~\ref{main 2} does not hold for $N=2$, since only half of the
 infinitesimal Hecke algebras are of the form given in the theorem (algebras $H_\kappa(\so_2,V_2)$ are the same as $H_{\kappa'}(\gl_1,V_1\oplus V_1^*)$).
\end{rem}

\subsection{$W$-algebras}
$\ $

 Here we recall the definitions of finite $W$-algebras following~\cite{GG} (see also~\cite[Section 1.5]{LT}).

  Let $\g$ be a finite dimensional simple Lie algebra over $\CC$ and $e\in\g$ be a nonzero nilpotent element.
 We identify $\g$ with $\g^*$ via the Killing form $(\ ,\ )$. Let $\chi$ be the element of $\g^*$ corresponding to $e$ and $\zz_{\chi}$
 be the stabilizer of $\chi$ in $\g$ (which is the same as the centralizer of $e$ in $\g$). Fix an $\mathfrak{sl}_2$-triple $(e,h,f)$ in $\g$.
 Then $\zz_\chi$ is $\ad(h)$--stable and the eigenvalues of $\ad(h)$ on $\zz_{\chi}$ are nonnegative integers.
  Consider the $\ad(h)$--weight grading on $\g=\bigoplus_{i\in\ZZ}\g(i)$, that is, $\g(i):=\{\xi\in\g| [h,\xi]=i\xi\}$.
 Equip $\g(-1)$ with the symplectic form $\omega_\chi(\xi,\eta):=\langle\chi,[\xi,\eta]\rangle$. Fix a Lagrangian subspace $l\subset \g(-1)$ and set
 $\m:=\bigoplus_{i\leq -2}\g(i)\oplus l\subset \g,\ \m':=\{\xi-\langle\chi,\xi\rangle|\xi\in\m\} \subset \U(\g)$.

\begin{defn}\label{W}\cite{P1,GG}
  The $W$-algebra associated with $e$ (and $l$) is the algebra $\U(\g,e):=\left(\U(\g)/\U(\g)\m' \right)^{\ad \m}$
 with multiplication induced from $\U(\g)$.
\end{defn}

  Let $\{F^{st}_{\bullet}\}$ denote the PBW filtration on $\U(\g)$, while $\U(\g)(i):=\{x\in \U(\g)| [h, x]=ix\}$.
 Define $F_k\U(\g)=\sum_{i+2j\leq k} (F^{st}_j \U(\g)\cap \U(\g)(i))$ and equip $\U(\g,e)$ with the induced filtration,
 denoted $\{F_{\bullet}\}$ and referred to as the {\it Kazhdan} filtration.

  One of the key results of~\cite{GG,P1} is a description of the associated graded algebra $\gr_{F_{\bullet}}\U(\g,e)$.
 Recall that the affine subspace $\s_e:=\chi+(\g/[\g,f])^*\subset\g^*$ is called the {\it Slodowy slice}.
 As an affine subspace of $\g$, the Slodowy slice $\s_e$ coincides with $e+\C$, where $\C=\Ker_{\g}\ad(f)$.
 So we can identify $\CC[\s_e]\cong \CC[\C]$ with the symmetric algebra $S(\zz_{\chi})$.
  According to~\cite[Section 3]{GG}, algebra $\CC[\s_e]$ inherits a Poisson structure from $\CC[\g^*]$ and is also graded with $\deg(\zz_{\chi}\cap \g(i))=i+2$.

\begin{thm}\label{gr}\cite[Theorem 4.1]{GG}
  The filtered algebra $\U(\g,e)$ does not depend on the choice of $l$ (up to a distinguished isomorphism) and
 $\gr_{F_{\bullet}} \U(\g,e)\cong \CC[\s_e]$ as graded Poisson algebras.
\end{thm}


$\ $

\section{Proof of Theorem~\ref{main 1}}

\noindent
 $\bullet$ \emph{Sufficiency.}

  Given any $c\in (\Oo(S)^*)^{\mathrm{SO}_N}$, the formula $\kappa(x,y):=((g-g^{-1})x,y)c$ defines a skew-symmetric  $\mathrm{SO}_N$-equivariant pairing
 $\kappa: V_N\times V_N\to \Oo(\mathrm{SO}_N)^*$. For $x,y,z\in V_N$ and $g\in \mathrm{SO}_N$ we define
   $$h(x,y,z;g):=(z-z^g)(x^g-x^{g^{-1}},y)+(y-y^g)(z^g-z^{g^{-1}},x)+(x-x^g)(y^g-y^{g^{-1}},z).$$

\begin{lem}
 We have $h(x,y,z;g)=0$ for all $x,y,z\in V_N$ and $g\in S$.
\end{lem}

\begin{proof}$\ $

  For any $g\in S$ consider the decomposition $V=V^g\oplus (V^g)^\perp$, where $V^g:=\mathrm{Ker}(1-g)$ is a codimension $\leq 2$ subspace of $V$.
 If either of the vectors $x,y,z$ belongs to $V^g$, then all the three summands are zero and the result follows.
 Thus, we can assume $x,y,z\in (V^g)^\perp$.
  Without loss of generality, we can assume that $z=\alpha x+\beta y$ with $\alpha,\beta\in \CC$, since $\dim\ (V^g)^\perp\leq 2$.
 Then
  $$h(x,y,z;g)=\alpha\left((x-x^g)(x^g-x^{g^{-1}},y)+(x-x^g)(y^g-y^{g^{-1}},x)+(y-y^g)(x^g-x^{g^{-1}},x)\right)+$$
  $$\beta\left((y-y^g)(x^g-x^{g^{-1}},y)+(y-y^g)(y^g-y^{g^{-1}},x)+(x-x^g)(y^g-y^{g^{-1}},y)\right).$$
 Clearly, $(x^g-x^{g^{-1}},x)=(x^g,x)-(x,x^g)=0$ and $(x^g-x^{g^{-1}},y)=-(y^g-y^{g^{-1}},x)$,
 so that the first sum is zero. Likewise, the second sum is zero. The result follows.
\end{proof}

  Since $c$ is scheme-theoretically supported on $S$, we get $h(x,y,z;g)c=0$ and so (\dag) holds.

\medskip
\noindent
 $\bullet$ \emph{Necessity.}





   Let $I\subset \CC[\mathrm{SO}_N]$ be the defining ideal of $\Phi$, that is, $I$ is generated by $3\times3$ determinants of $1-g$.
 Consider a closed subscheme $\bar{\Phi}\subset \so_N$, defined by the ideal $\bar{I}:=(\wedge^3A)\subset \CC[\so_N]$.


  Define $E:=\Rad(I)/I$ and $\bar{E}:=\Rad(\bar{I})/\bar{I}$.
  Notice that $\bar{E}\simeq E$, since $\Phi$ is reduced in the formal neighborhood of any
 point $g\ne 1$, while the exponential map defines an isomorphism of formal completions
 $\exp:\bar{\Phi}^{\wedge_0}\iso \Phi^{\wedge_1}$.

 On the other hand, we have a short exact sequence of $\mathrm{SO}_N$-modules
  $$0\to E\to \Oo(\Phi)\to \Oo(S)\to 0,$$
 inducing the following short exact sequence of vector spaces
\begin{equation}\tag{$\natural$}
   0\to (\wedge^2 V_N^*\otimes \Oo(S)^*)^{\mathrm{SO}_N}\overset{\phi}\to (\wedge^2 V_N^*\otimes \Oo(\Phi)^*)^{\mathrm{SO}_N}\overset{\psi}\to
     (\wedge^2 V_N^*\otimes E^*)^{\mathrm{SO}_N}\to 0.
\end{equation}

  It is easy to deduce the necessity for $\kappa\in \mathrm{Im}(\phi)$ by utilizing the arguments from the proof of~\cite[Theorem 3.14(ii)]{EGG}.
 Combining this observation with Proposition~\ref{0} and an isomorphism $E\simeq \bar{E}$, it suffices to prove the following result:

\begin{lem}\label{fun}
 (a) The space $(\wedge^2 V_N^*\otimes \bar{E}^*)^{\mathrm{SO}_N}$ is either zero or one-dimensional.

\noindent
 (b) If $(\wedge^2 V_N^*\otimes \bar{E}^*)^{\mathrm{SO}_N}\ne 0$, then there exists $\kappa'\in (\wedge^2 V_N^*\otimes
 \Oo(\Phi)^*)^{\mathrm{SO}_N}$ not satisfying (\dag).\footnote{\ So that any element of
 $(\wedge^2 V_N^*\otimes \Oo(\Phi)^*)^{\mathrm{SO}_N}$ satisfying (\dag) should be in the image of $\phi$.}
\end{lem}

  Notice that the adjoint action of $\mathrm{SO}_N$ on $\so_N$
 extends to the action of $\mathrm{GL}_N$ by $g_\cdot A=gAg^t$ for $A\in \so_N,g\in \mathrm{GL}_N$.
 This endows $\CC[\so_N]$ with a structure of a $\mathrm{GL}_N$-module and both $\bar{I},\Rad(\bar{I})$ are $\mathrm{GL}_N$-invariant.
 The following fact was communicated to us by Steven Sam:

\begin{claim}\label{radical}
 As $\gl_N$-representations $\bar{E}\simeq \wedge^4 V_N$.
\end{claim}

 Let us first deduce Lemma~\ref{fun} from this Claim.

\begin{proof}[Proof of Lemma~\ref{fun}]
 $\ $

 (a) The following facts are well-known (see~\cite[Theorems 19.2, 19.14]{FH}):

\noindent
 $\circ$ the $\so_{2n+1}$-representations $\{\wedge^i V_{2n+1}\}_{i=0}^n$ are irreducible and pairwise non-isomorphic,

\noindent
 $\circ$ the $\so_{2n}$-representation $\wedge^nV_{2n}$ decomposes as $\wedge^nV_{2n}\simeq \wedge^n_+V_{2n}\oplus \wedge^n_- V_{2n}$,
         and $\so_{2n}$-representations $\{\wedge^0V_{2n},\ldots,\wedge^{n-1}V_{2n},\wedge^n_+V_{2n},\wedge^n_- V_{2n}\}$
         are irreducible and pairwise non-isomorphic.

 Combining these facts with Claim~\ref{radical} and an isomorphism $\wedge^kV_N\simeq \wedge^{N-k}V_N^*$, we get
 $$(\wedge^2 V_{2n+1}^*\otimes \bar{E}^*)^{\mathrm{SO}_{2n+1}}=0,\ \
 \mathrm{while}\ \ \dim((\wedge^2 V_{2n}^*\otimes \bar{E}^*)^{\mathrm{SO}_{2n}})=
   \left\{\begin{array}{cc}
    1\ , & n=3 \\
    0\ , & n\ne 3 \\
   \end{array}
  \right. .$$

 (b) For $N=6$, any nonzero element of $(\wedge^2 V_6^*\otimes \bar{E}^*)^{\mathrm{SO}_6}$ corresponds to the composition
  $$\wedge^2 V_6\underset{\varphi}\iso \wedge^4V_6^* \simeq \bar{E}^*.$$

  Let $M_4\subset \CC[\so_N]_2$ be the subspace spanned by the Pfaffians of all $4\times 4$ principal minors.
 This subspace is $\mathrm{GL}_6$-invariant and $M_4\simeq \wedge^4V_6$ as $\gl_6$-representations.
 Claim~\ref{radical} and simplicity of the spectrum of the $\gl_6$-module $\CC[\so_6]$ (see Theorem~\ref{AF-theorem} below) imply
 $M_4\subset \Rad(\bar{I})$ and $M_4\cap \bar{I}=0$.
  It follows that $M_4$ corresponds to the copy of $\wedge^4V_6\simeq \Rad(\bar{I})/\bar{I}$ from Claim~\ref{radical}.

 Choose an orthonormal basis $\{y_i\}_{i=1}^6$ of $V_6$, so that any element $A\in\so_6$ is skew-symmetric with respect to this basis.
  We denote the corresponding Pfaffian by $\Pf_{\widehat{i,j}}$
 (with a correctly chosen sign).\footnote{\ To make a compatible choice of signs,
    define $\Pf_{\widehat{i,j}}$ as the derivative of the total Pfaffian $\Pf$ along $E_{ij}-E_{ji}$.}
 We define $\kappa'(y_i\otimes y_j)\in U(\so_6)$ to be the symmetrization of $\Pf_{\widehat{i,j}}$.
 Identifying $U(\so_6)$ with $S(\so_6)$ as $\so_6$-modules,
 we easily see that $\kappa':\wedge^2 V_6\to U(\so_6)$ is $\so_6$-invariant.

  However, $\kappa'$ does not satisfy the Jacobi identity.
 Indeed, let us define $\bar{\kappa}':V_6\otimes V_6\to S(\so_6)$ by $\bar{\kappa'}(y_i\otimes y_j)=\Pf_{\widehat{i,j}}$.
 Then for any three different indices $i,j,k$, the corresponding expressions
   $\{P_{\widehat{i,j}},x_k\}, \{P_{\widehat{j,k}},x_i\}, \{P_{\widehat{k,i}},x_j\}$ coincide up to a sign and are nonzero.
 So their sum is also non-zero, implying that (\dag) fails for $\kappa'$.
\end{proof}

\medskip
\noindent
 $\bullet$ \emph{Proof of Claim~\ref{radical}}.

\medskip
\noindent
 $\circ$ \emph{Step 1: Description of $\Rad(\bar{I})$}.

  Let $\Pf_{ijkl}\in \CC[\so_N]_2$ be the Pfaffians of the principal $4\times 4$ minors corresponding to the rows/columns $\# i,j,k,l$.
 It is clear that $\Pf_{ijkl}$ vanish at rank $\leq 2$ matrices and so $\Pf_{ijkl}\in \Rad(\bar{I})$.
 A beautiful classical result states that those elements generate $\Rad(\bar{I})$, in fact:

\begin{thm}\label{Weyman}~\cite[Theorem 6.4.1(b)]{W}
 The ideal $\Rad(\bar{I})$ is generated by $\{\Pf_{ijkl}|i<j<k<l\}$.
\end{thm}

\medskip
\noindent
 $\circ$ \emph{Step 2: Decomposition of $\CC[\so_N]$ as a $\gl_N$-module}.

  Let $T$ be the set of all length $\leq N$ Young diagrams $\lambda=(\lambda_1\geq \lambda_2\geq \cdots\geq 0)$.
 There is a natural bijection between $T$ and the set of all irreducible finite dimensional polynomial $\gl_N$-representations.
 For $\lambda\in T$, we denote the corresponding irreducible $\gl_N$-representation by $L_\lambda$.
 Let $T^e$ be the subset of $T$ consisting of all Young diagrams with even columns.

  The following result describes the decomposition of $\CC[\so_N]$ into irreducibles:

\begin{thm}\cite[Theorem 2.5]{AF}\label{AF-theorem}
 As $\gl_N$-representations $\CC[\so_N]\simeq S(\wedge^2 V_N) \simeq \bigoplus_{\lambda\in T^e} L_\lambda$.
\end{thm}

  For any $\lambda\in T^e$, let $\mathcal{I}_\lambda\subset \CC[\so_N]$ be the ideal generated by $L_\lambda\subset \CC[\so_N]$,
 while $T^e_{\lambda}\subset T^e$ be the subset of the diagrams containing $\lambda$.
 The arguments of~\cite{AF} (see also~\cite[Theorem 5.1]{D}) imply that $\mathcal{I}_\lambda\simeq \bigoplus_{\mu\in T^e_\lambda} L_\mu$ as $\gl_N$-modules.

\medskip
\noindent
 $\circ$ \emph{Step 3: $\Rad(\bar{I})$ and $\bar{I}$ as $\gl_N$-representations}.

  Since the subspace $M_4\subset \CC[\so_N]$, spanned by $\Pf_{ijkl}$, is $\gl_N$-invariant and is isomorphic to $\wedge^4 V_N$,
 the results of the previous steps imply that $\Rad(\bar{I})\simeq \bigoplus _{\mu\in T^e_{(1^4)}} L_{\mu}$ as $\gl_N$-modules.

  Let $N_3\subset \CC[\so_N]_3$ be the subspace spanned by the determinants of all $3\times 3$ minors.
 This is a $\gl_N$-invariant subspace.

\begin{lem}
 We have $N_3\simeq L_{(2^2,1^2)}\oplus L_{(1^6)}$ as $\gl_N$-representations.
\end{lem}

\begin{proof}$\ $

  According to Step 2, we have $\CC[\so_N]_3\simeq L_{(1^6)}\oplus L_{(2^2,1^2)}\oplus L_{(3^2)}$.
 Since the space of $3\times 3$ minors identically vanishes when $N=2$, and the Schur functor (3,3) does not, it rules $L_{(3^2)}$ out.
 Also, the space of $3\times 3$ minors is nonzero for $N=4$, while the Schur functor $(1^6)$ vanishes, so $N_3\ncong L_{(1^6)}$.
 Since partition $(1^6)$ corresponds to the subspace $M_6\subset \CC[\so_N]$ spanned by $6\times 6$ Pfaffians, it suffices to prove that $M_6\subset N_3$.
 The latter is sufficient to verify for $N=6$, that is, the Pfaffian $\Pf$ of a $6\times 6$ matrix is a linear combination of its $3\times 3$
 determinants.\footnote{\
   The conceptual proof of this fact is as follows. Note that determinants of $3\times 3$ minors of $A\in \so_6$ are just
  the matrix elements of $\wedge^3 A$, and $\wedge^3 A$ acts on $\wedge^3V_6=\wedge^3_+V_6\oplus \wedge^3_-V_6$.
  It is easy to see that the trace of $\wedge^3A$ on $\wedge^3_+V_6$ is nonzero.
   This provides a cubic invariant for $\so_6$, which is unique up to scaling (multiple of $\Pf$).
 }

  Let $\det_{ijk}^{pqs}$ be the determinant of the $3\times 3$ minor, obtained by intersecting
 rows $\#i,j,k$ and columns $\# p,q,s$. The following identity is straightforward:
  $$-4\Pf=-\mathrm{det}_{123}^{456}+\mathrm{det}_{124}^{356}-\mathrm{det}_{125}^{346}+\mathrm{det}_{126}^{345}-\mathrm{det}_{134}^{256}+\mathrm{det}_{135}^{246}-
    \mathrm{det}_{136}^{245}-\mathrm{det}_{145}^{236}+\mathrm{det}_{146}^{235}-\mathrm{det}_{156}^{234}.$$
  This completes the proof of the lemma.
\end{proof}

  The results of Step 2 imply that $\bar{I}\simeq \bigoplus_{\mu\in T^e_{(2^2,1^2)}\cup T^e_{(1^6)}} L_\mu$ as $\gl_N$-modules.

\medskip
 Claim~\ref{radical} follows from the aforementioned descriptions of $\gl_N$-modules $\bar{I}$ and $\Rad (\bar{I})$. \ \  $\blacksquare$


%
%
%


$\ $

\section{Proof of Theorem~\ref{main 2}}

\medskip
 Let us introduce some notation:

\medskip
\noindent
 $\bullet$ $K:=\mathrm{SO}_N(\RR)$ (the maximal compact subgroup of $G=\mathrm{SO}_N(\CC)$),

\medskip
\noindent
 $\bullet$ $s_\theta=\left(%
\begin{array}{ccccc}
  \cos\theta & -\sin\theta & 0 & \cdots & 0 \\
  \sin\theta & \cos\theta & 0 & \cdots & 0 \\
  0 & 0 & 1 & \cdots & 0 \\
  \vdots & \vdots & \vdots & \ddots & \vdots \\
  0 & \cdots & 0 & \cdots & 1 \\
\end{array}%
\right)\in K$,\ \ \ $\theta\in [-\pi,\pi]$,

\medskip
\noindent
 $\bullet$ $S_\theta:=\{gs_\theta g^{-1}|g\in K\}\subset K$,

\medskip
\noindent
 $\bullet$ $S_\RR:=S\cap K=\bigcup_{\theta\in [0,\pi]}S_\theta$,\footnote{\ Note that $S_\theta$ and $S_{-\theta}$ coincide for $N\geq 3$. That explains why $\theta\in [0,\pi]$ instead of $\theta\in [-\pi,\pi]$.}
 so that $S_\RR/K$ gets identified with $S^1/\ZZ_2$.

\medskip
 According to Theorem~\ref{main 1}, there exists a $\ZZ_2$-invariant $c\in \Oo_0(S^1)^*$, which is
 a linear combination of the delta-function $\delta_0$ (at $0\in S^1$) and its even derivatives $\delta_0^{(2k)}$,
 such that\footnote{\ Here we integrate over the whole circle $S^1$ instead of $S^1/\ZZ_2$, but we require $c(\theta)=c(-\theta)$.}
    $$\kappa(x,y)=\int_{-\pi}^\pi c(\theta)\left(\int_{S_\theta}((g-g^{-1})x,y)\ dg\right) d\theta \ \ \ \mathrm{for\ all}\ x,y\in V_N.$$

  For $g\in S_\RR$ we define a $2$-dimensional subspace $V_g\subset V_N$ by $V_g:=\mathrm{Im}(1-g)$.
  To evaluate the above integral, choose length $1$ orthogonal vectors $p, q\in V_g$ such that the restriction of $g$ to $V_g$ is given by the matrix
 $\left(%
\begin{array}{cc}
  \cos\theta & -\sin\theta \\
  \sin\theta & \cos\theta \\
\end{array}%
\right)$
 in the basis $\{p,q\}$.

 Let us define $J_{p,q}:=q\otimes p^t-p\otimes q^t\in \so_N(\RR)$. We have:

\medskip
\noindent
 $\bullet$ $((g-g^{-1})x,y)= 2\sin \theta\cdot (x,J_{p,q}y)$,

\medskip
\noindent
 $\bullet$ $g=\exp(\theta J_{p,q})$, since
 $\left(%
\begin{array}{cc}
  \cos\theta & -\sin\theta \\
  \sin\theta & \cos\theta \\
\end{array}%
\right)= \exp\left(\theta\cdot
\left(%
\begin{array}{cc}
  0 & -1 \\
  1 & 0 \\
\end{array}%
\right)\right)$.

\medskip
  As a result, we get:\footnote{\ Generally speaking, the integration should be taken over the Grassmannian $G_2(\RR^N)$.
    However, it is easier to integrate over the
   Stiefel manifold $V_2(\RR^N)$, which is a principal $\mathrm{O}(2)$-bundle over $G_2(\RR^N)$.}
\begin{equation}\label{commutator}
 \kappa(x,y)=\int_{p\in S^{N-1}}\int_{q\in S^{N-2}(p)}(x,J_{p,q}y)\left(\int_{-\pi}^\pi 2 c(\theta)\sin\theta\cdot e^{\theta J_{p,q}}\ d\theta\right)\ dqdp,
\end{equation}
 where $S^{N-1}$ is the unit sphere in $\RR^N$ centered at the origin and
 $S^{N-2}(p)$ is the unit sphere in $\RR^{N-1}(p)\subset \RR^N$, the hyperplane orthogonal to the line passing through $p$ and the origin.

   Since $c(\theta)$ is an arbitrary linear combination of the delta-function and its even derivatives,
 the above integral is a linear combination of the following integrals:
 $$\int_{p\in S^{N-1}}\int_{q\in S^{N-2}(p)}(x,J_{p,q}y)\cdot J_{p,q}^{2k+1}\ dq\ dp,\ \ k\geq 0.$$

   This is a standard integral (see ~\cite[Section 4.2]{EGG} for the analogous calculations).
 Identifying $U(\so_N)$ with $S(\so_N)$ via the symmetrization map, it suffices to compute the integral
  $$I_{m;x,y}(A)=\int_{p\in S^{N-1}}\int_{q\in S^{N-2}(p)}(x,J_{p,q}y)\cdot \tr(AJ_{p,q})^m\ dq\ dp,\ \ \ A\in \so_N(\RR).$$

 To compute this expression we introduce
  $$F_m(A):=\int_{p\in S^{N-1}}\int_{q\in S^{N-2}(p)}\tr(AJ_{p,q})^{m+1}\ dq\ dp=\int_{p\in S^{N-1}}\int_{q\in S^{N-2}(p)}(2(Aq,p))^{m+1}\ dq\ dp,$$
 so that the former integral can be expressed in the following way:
  $$dF_m(A)(x\otimes y^t-y\otimes x^t)=-2(m+1)I_{m;x,y}(A).$$

 Now we compute $F_m(A)$. Notice that
  $$G_m(A,\z):=\int_{p\in \RR^N}\int_{q\in \RR^{N-1}(p)}(2(Aq,p))^{m+1}e^{-\z(p,p)-\z(q,q)}\ dq\ dp=$$
  $$\int_0^\infty \int_0^\infty e^{-\z r_1^2-\z r_2^2}\int_{|p|=r_1}\int_{|q|=r_2}(2(Aq,p))^{m+1}\ dq\ dp\ dr_2\ dr_1=$$
  $$\int_0^\infty \int_0^\infty e^{-\z r_1^2-\z r_2^2}r_1^{m+N}r_2^{m+N-1}\ dr_2dr_1\cdot F_m(A)=K_{m+N}(\z)K_{m+N-1}(\z)F_m(A),$$
   where
  $$K_l(\z):=\int_0^\infty e^{-\z r^2}r^l\ dr=
    \left\{
    \begin{array}{lr}
      \frac{k!}{2\z^{k+1}}\ , & \ \ \ l=2k+1\\
      \frac{(2k-1)!!\sqrt{\pi}}{2^{k+1}\z^{k+1/2}}\ , & l=2k
    \end{array}
    \right.
  .$$

\noindent
   As a result, we get
  $$G_m(A,\z)=\frac{\sqrt{\pi}(m+N-1)!}{2^{m+N+1}\z^{m+N+1/2}}F_m(A).$$

\noindent
   On the other hand, we have:
  $$\sum_{m=-1}^\infty \frac{1}{(m+1)!}G_m(A,\z)=\int_{p\in \RR^N}\int_{q\in \RR^{N-1}(p)}e^{2(Aq,p)}e^{-\z(p,p)-\z(q,q)}\ dq\ dp=$$
  $$\int_{p\in \RR^N}e^{-\z (p,p)} \int_{q\in \RR^{N-1}(p)}e^{-2(q,Ap)-\z(q,q)}\ dq\  dp\overset{q':=q+\frac{Ap}{\z}}=$$
  $$\int_{p\in \RR^N}e^{-\z (p,p)} \int_{q'\in \RR^{N-1}(p)}e^{-\z(q',q')}e^{\frac{1}{\z}(Ap,Ap)}\ dq'\ dp=
    \int_{p\in \RR^N}e^{-\z (p,p)+\frac{1}{\z}(Ap,Ap)}\ dp\cdot(\pi/\z)^{\frac{N-1}{2}}=$$
  $$(\pi/\z)^{\frac{N-1}{2}}\int_{p\in \RR^N}e^{((-\z-\frac{1}{\z}A^2)p,p)}\ dp=
    \frac{\pi^{N-\frac{1}{2}}}{\z^{\frac{N-1}{2}}}\det\left(\z+\frac{1}{\z}A^2\right)^{-1/2}=
    \frac{\pi^{N-\frac{1}{2}}}{\z^{N-\frac{1}{2}}}\det(1+\z^{-2}A^2)^{-1/2}.$$

  Hence, $F_m(A)$ is equal to a constant times the coefficient of $\tau^{m+1}$ in  $\det(1+\tau^2A^2)^{-1/2}$, expanded as a power series in $\tau$.
 Differentiating $\det(1+\tau^2A^2)^{-1/2}$ along $B\in \so_N$, we get
  $$\frac{\partial}{\partial B}\left(\det(1+\tau^2A^2)^{-1/2}\right)=-\frac{\tau^2\tr(BA(1+\tau^2A^2)^{-1})}{\det(1+\tau^2A^2)^{1/2}}.$$

\noindent
  Setting $B=x\otimes y^t-y\otimes x^t$ yields $2\tau^2(x,A(1+\tau^2A^2)^{-1}y)\det(1+\tau^2A^2)^{-1/2}$ as
  desired. \ \ \ $\blacksquare$


$\ $

\section{Poisson center of algebras $H_\z^{\cl}(\so_N)$}

  Following~\cite{DT}, we introduce the Poisson algebras $H_\z^{\cl}(\so_N,V_N)$, where $\z=(\z_0,\ldots,\z_k)$ is a deformation parameter.
 As algebras these are $S(\so_N\oplus V_N)$ with a Poisson bracket $\{\cdot,\cdot\}$ modeled after the commutator
 $[\cdot,\cdot]$ of $H_\z(\so_N,V_N)$, that is, $\{x,y\}=\sum_{j}\z_j\gamma_{2j+1}(x,y)$.
  We prefer the following short formula for $\{\cdot,\cdot\}:V_N\times V_N\rightarrow \CC[\so_N]\simeq S(\so_N)$:
\begin{equation}\label{good formula} \tag{*}
 \{x,y\}=\Res_{z=0}
 \z(z^{-2})(x,A(1+z^2A^2)^{-1}y)\det(1+z^2A^2)^{-1/2}z^{-1}dz,\
 \ \forall\ x,y\in V_N, A\in \so_N,
\end{equation}
 where $\z(z):=\sum_{i\geq 0}\z_iz^i$ is the generating function of the deformation parameters.

  In fact, we can view algebras $H_\z(\so_N,V_N)$ as \emph{quantizations} of the algebras $H_\z^{\cl}(\so_N,V_N)$.
 The latter algebras still carry some important information.
  The main result of this section is a computation of their Poisson center $\zz_{\mathrm{Pois}}(H^{\cl}_\z(\so_N,V_N))$.

  Let us first recall the corresponding result in the non-deformed
 case ($\z=0$), when the corresponding algebra is just $S(\so_N\ltimes V_N)$ with a Lie-Poisson bracket.
 To state the result we introduce some more notation:

\medskip
\noindent
 $\bullet$ Define $p_i(A)\in \CC$ via $\det(I_N+tA)=\sum_{j=0}^N p_j(A)t^j$ for $A\in \gl_N$.

\noindent
 $\bullet$ Define $b_i(A)\in \gl_N$ via $b_0(A)=I_N,\ b_k(A)=\sum_{j=0}^k (-1)^jp_j(A)A^{k-j}$ for $k>0$.

\noindent
 $\bullet$ Define $\aaa_N:=\so_N\ltimes V_N$; we identify $\aaa_N^*$ with $\aaa_N$ via the natural pairing.

\noindent
 $\bullet$ Define $\psi_k:\aaa_N^*\rightarrow \CC$ by $\psi_k(A,v)=(v,b_{2k}(A)v)$ for $A\in \so_N,\ v\in V_N,\ k\geq 0$.

\noindent
 $\bullet$ If $N=2n+1$, $\psi_n$ is actually the square of a polynomial function
 $\widehat{\psi}_n$, which can be realized explicitly as the Pfaffian of
 the matrix
 $\left(%
\begin{array}{cc}
  A & v \\
  -v^t & 0 \\
\end{array}%
\right)\in \so_{2n+2}$.

\noindent
 $\bullet$ Identifying $\CC[\aaa_N^*]\simeq S(\aaa_N)$, let $\tau_k\in S(\aaa_N)$ (respectively $\widehat{\tau}_{n+1}\in S(\aaa_{2n+1})$)
 be the elements corresponding to $\psi_{k-1}$ (respectively $\widehat{\psi}_n$).

\medskip
  The following result is due to~\cite[Sections 3.7, 3.8]{R}:

\begin{prop}\label{non-deformed center}
 Let $\zz_{\mathrm{Pois}}(A)$ denote the Poisson center of the Poisson algebra $A$. We have:

\noindent
 (a) $\zz_{\mathrm{Pois}}(S(\aaa_{2n}))$ is a polynomial algebra in free generators $\{\tau_1,\ldots,\tau_n\}$.

\noindent
 (b) $\zz_{\mathrm{Pois}}(S(\aaa_{2n+1}))$ is a polynomial algebra in free generators $\{\tau_1,\ldots,\tau_n,\widehat{\tau}_{n+1}\}$.
\end{prop}

  Similarly to the cases of $\gl_n,\spn_{2n}$, this result can be generalized for arbitrary deformations $\z$.
 In fact, for any deformation parameter $\z=(\z_0,\ldots,\z_k)$ the Poisson center
 $\zz_{\mathrm{Pois}}(H^{\cl}_\z(\so_N,V_N))$ is still a polynomial algebra in $\lfloor\frac{N+1}{2}\rfloor$ generators.
  This is established in the following theorem:

\begin{thm}\label{deformed case}
  Define $c_i\in \CC[\so_N]^{\mathrm{SO}_N}\simeq \zz_{\mathrm{Pois}}(S(\so_N))$
  via $\sum_i(-1)^ic_it^{2i}=c(t)$, where
 $$c(t):=\Res_{z=0}\ \z(z^{-2})\frac{\det(1+t^2A^2)^{1/2}}{\det(1+z^2A^2)^{1/2}}\frac{z^{-1}dz}{1-t^{-2}z^2}.$$
\noindent
 (a) $\zz_{\mathrm{Pois}}(H^{\cl}_\z(\so_{2n},V_{2n}))$ is a polynomial algebra in free generators $\{\tau_1+c_1,\ldots,\tau_n+c_n\}$.

\noindent
 (b) $\zz_{\mathrm{Pois}}(H^{\cl}_\z(\so_{2n+1},V_{2n+1}))$ is a polynomial algebra in free generators $\{\tau_1+c_1,\ldots,\tau_n+c_n,\widehat{\tau}_{n+1}\}$.
\end{thm}


 Let us introduce some more notation before proceeding to the proof:

\noindent
 $\bullet$ Let $\{x_i\}_{i=1}^N$ be a basis of $V_N$ such that $(x_i,x_j)=\delta_{N+1-i}^j$.

\noindent
 $\bullet$ Let $J=(J_{ij})_{i,j=1}^N$ be the corresponding \emph{anti-diagonal} symmetric matrix, i.e., $J_{ij}=\delta_{N+1-i}^j$.

 Notice that $A=(a_{ij})\in \so_N$ if and only if $a_{ij}=-a_{N+1-j,N+1-i}$ for all $i,j$.

\noindent
 $\bullet$ Let $\h_N$ be the Cartan subalgebra of $\so_N$ consisting of the diagonal matrices.

\noindent
 $\bullet$ Define $e_{(i,j)}:=E_{i,j}-E_{N+1-j,N+1-i}\in \so_N$ for $i,j\leq N$ (in particular, $e_{(i,N+1-i)}=0\ \forall i$).

\noindent
 $\bullet$ We set $e_i:=e_{(i,i)}$ for $1\leq i\leq n:=\lfloor \frac{N}{2} \rfloor$, so that $\{e_i\}_{i=1}^n$ form a basis of $\h_N$.

\noindent
 $\bullet$ Define symmetric polynomials $\sigma_i\in \CC[z_1,\ldots,z_n]^{S_n}$ via $\prod_{i=1}^n(1+tz_i)=\sum_{i=0}^n t^i\sigma_i(z_1,\ldots,z_n)$.

\medskip
\noindent
 \emph{Proof of Theorem~\ref{deformed case}.}
 $\ $

  We shall show that the elements $\tau_i+c_i$ (and $\widehat{\tau}_{n+1}$ for $N=2n+1$) are Poisson central.
 Combined with Proposition~\ref{non-deformed center} this clearly
 implies the result by a deformation argument.
  Since $\{\tau_i,\so_N\}=0$ for $\z=0$, we still have $\{\tau_i,\so_N\}=0$ for arbitrary $\z$.
 This implies $\{\tau_i+c_i,\so_N\}=0$ as  $c_i\in \zz_{\mathrm{Pois}}(S(\so_N))$. Therefore we just need to verify
\begin{equation}\label{need1}
 \{c_i,x_q\}=-\{\tau_i,x_q\}\ \ \ \mathrm{for\ all}\ 1\leq q\leq N.
\end{equation}

  Using $\psi_s(A,v)=(v,b_{2s}(A)v)=\sum_{k,l=1}^N x_kx_l{b_{2s}(A)}_{N+1-k,l}$, we get:
  $$\{\tau_{s+1},x_q\}=\sum_{k,l}\{{b_{2s}(A)}_{N+1-k,l},x_q\}x_kx_l+\sum_{k,l}{b_{2s}(A)}_{N+1-k,l}\{x_k,x_q\}x_l+\sum_{k,l}{b_{2s}(A)}_{N+1-k,l}x_k\{x_l,x_q\}.$$

 The first summand is zero due to Proposition~\ref{non-deformed center}.
  On the other hand, $AJ+JA^t=0$ implies $(A^{2j})_{N+1-k,l}=(A^{2j})_{N+1-l,k}$ and $p_{2j+1}(A)=0$ for all $j\geq 0$. Hence,
 $$b_{2s}(A)=A^{2s}+p_2(A)A^{2s-2}+p_4(A)A^{2s-4}+\ldots+p_{2s}(A),\  {b_{2s}(A)}_{n+1-k,l}=b_{2s}(A)_{n+1-l,k}.$$

 Combining this with $\{c_{s+1},x_q\}=\sum_{p\ne N+1-q}\frac{\partial c_{s+1}}{\partial e_{(p,q)}}x_p$, we see that~(\ref{need1}) is equivalent to:
\begin{equation}\label{need2}
 \frac{\partial c_{s+1}}{\partial e_{(p,q)}}=
 -2\sum_{l}{b_{2s}(A)}_{N+1-p,l}\Res_{z=0} \z(z^{-2})\frac{(x_l,A(1+z^2A^2)^{-1}x_q)}{\det(1+z^2A^2)^{1/2}}\frac{dz}{z}\ \ \ \mathrm{for\ all}\ p,q\leq N.
\end{equation}

  Because both sides of~(\ref{need2}) are $\mathrm{SO}_N$-invariant, it suffices to verify~(\ref{need2}) for $A\in \h_N$, that is, for

\noindent
 $\bullet$
  $A=\diag(\lambda_1,\ldots,\lambda_n,-\lambda_n,\ldots,-\lambda_1)$ in the case $N=2n$,

\noindent
 $\bullet$  $A=\diag(\lambda_1,\ldots,\lambda_n,0,-\lambda_n,\ldots,-\lambda_1)$ in the case $N=2n+1$.

  For $p\ne q$, both sides of~(\ref{need2}) are zero.
 For $p=q\leq n$, the only nonzero summand on the right hand side of~(\ref{need2}) is the one corresponding to $l=N+1-q$.
 In this case:
   $${b_{2s}(A)}_{N+1-q,N+1-q}=\lambda_q^{2s}-\sigma_1(\lambda_1^2,\ldots,\lambda_n^2)\lambda_q^{2s-2}+\ldots+(-1)^s\sigma_s(\lambda_1^2,\ldots,\lambda_n^2)=
     (-1)^s\frac{\partial \sigma_{s+1}(\lambda_1^2,\ldots,\lambda_n^2)}{\partial \lambda_q^2},$$
 while $(x_{N+1-q},A(1+z^2A^2)^{-1}x_q)=\frac{\lambda_q}{1+z^2\lambda_q^2}$ and $\det(1+z^2A^2)^{1/2}=\prod_{i=1}^n (1+z^2\lambda_i^2)$.

\noindent
 For $p=q>\lfloor\frac{N+1}{2}\rfloor$, we get the same equalities with $\lambda_i\leftrightarrow -\lambda_i$. As a result,~(\ref{need2}) is equivalent to:
\begin{equation*}\label{need3}
 \frac{\partial c_{s+1}(\lambda_1,\ldots,\lambda_n)}{\partial \lambda_q^2}=
 (-1)^{s+1}\frac{\partial \sigma_{s+1}(\lambda_1^2,\ldots,\lambda_n^2)}{\partial \lambda_q^2}
 \Res_{z=0} \z(z^{-2})\frac{z^{-1}dz}{(1+z^2\lambda_q^2)\prod_{i=1}^n (1+z^2\lambda_i^2)}.
\end{equation*}

  We thus need to verify the following identities for $c(t)$:
\begin{equation}\label{need4}
 \frac{\partial c(t)}{\partial \lambda_q^2}=\frac{\partial \prod_{i=1}^n(1+t^2\lambda_i^2)}{\partial \lambda_q^2}
    \Res_{z=0} \frac{\z(z^{-2})z^{-1}dz}{(1+z^2\lambda_q^2)\prod_{i=1}^n(1+z^2\lambda_i^2)}.
\end{equation}


  This is a straightforward verification and we leave it to an interested reader.
  This proves that $\tau_i+c_{i}\in  \zz_{\mathrm{Pois}}(H^{\cl}_\z(\so_N,V_N))$ for all $1\leq i\leq  n$.
 For $N=2n+1$, we also get a Poisson-central element  $\tau_{n+1}+c_{n+1}$. Since $c_{n+1}=0$, we
 have
 $$\widehat{\tau}_{n+1}^2=\tau_{n+1}\in \zz_{\mathrm{Pois}}(H^{\cl}_\z(\so_{2n+1},V_{2n+1}))\Rightarrow
   \widehat{\tau}_{n+1}\in \zz_{\mathrm{Pois}}(H^{\cl}_\z(\so_{2n+1},V_{2n+1})).$$
 This completes the proof of the theorem. \ \ \ $\blacksquare$

\begin{defn}
 The element $\tau_1'=\tau_1+c_1$ is called the \emph{Poisson Casimir element} of $H^{\cl}_\z(\so_N,V_N)$.
\end{defn}

 As a straightforward consequence of Theorem~\ref{deformed case}, we get:

\begin{cor}\label{Casimir Poisson}
 We have $\tau_1'=\tau_1+\sum_{j=0}^k (-1)^{j+1}\z_j\tr S^{2j+2}A$.
\end{cor}


$\ $

\section{The key isomorphism}

\subsection{Algebras $H_m(\so_N,V_N)$}
$\ $

 Let us first introduce the universal infinitesimal Hecke algebras of $(\so_N,V_N)$:

\begin{defn}
  Define the \emph{universal length $m$ infinitesimal Hecke algebra} $H_m(\so_N,V_N)$ as
\begin{equation*}
 H_m(\so_N,V_N):=\U(\so_N)\ltimes T(V_N)[\z_0,\ldots,\z_{m-1}]/([A,x]-A(x), [x,y]-\sum_{j=0}^{m-1} \z_j r_{2j+1}(x,y)-r_{2m+1}(x,y)),
\end{equation*}
 where $A\in \so_N,\ x,y\in V_N$ and $\{\z_i\}_{i=0}^{m-1}$ are central. The filtration is induced from the grading on
 $T(\so_N\oplus V_N)[\z_0,\ldots,\z_{m-1}]$ with $\deg(\so_N)=2,\ \deg(V_N)=2m+2$ and $\deg(\z_i)=4(m-i)$.
\end{defn}

  The algebra $H_m(\so_N,V_N)$ is free over $\CC[\z_0,\ldots,\z_{m-1}]$ and $H_m(\so_N,V_N)\slash(\z_i-c_i)_{i=0}^{m-1}$ is the
 usual infinitesimal Hecke algebra $H_{\z_c}(\so_N,V_N)$ for $\z_c=c_0r_1+\ldots+c_{m-1}r_{2m-1}+r_{2m+1}$.

\begin{rem}\label{PBW_updated}
  For an $\so_N$-equivariant pairing $\eta:\wedge^2 V_N\rightarrow \U(\so_N)[\z_0,\ldots,\z_{m-1}]$ such that
 $\deg(\eta(x,y))\leq 4m+2$, the algebra $\U(\so_N)\ltimes T(V_N)[\z_0,\ldots,\z_{m-1}]/([A,x]-A(x), [x,y]-\eta(x,y))$
 satisfies the PBW property if and only if $\eta(x,y)=\sum_{i=0}^{m}{\eta_ir_{2i+1}(x,y)}$
 with $\eta_i\in \CC[\z_0,\ldots,\z_{m-1}]$ degree $\leq 4(m-i)$ polynomials (this is completely analogous to Theorem~\ref{main 2}).
\end{rem}

\subsection{Isomorphisms $\bar{\Theta}$ and $\bar{\Theta}^{\cl}$}
$\ $

  The main goal of this section it to establish an abstract isomorphism between the algebras $H_m(\so_N,V_N)$
 and the $W$-algebras $U(\so_{N+2m+1},e_m)$, where $e_m\in \so_{N+2m+1}$ is a nilpotent element of the Jordan type $(1^N,2m+1)$.
  We make a particular choice of such an element:\footnote {\ In this section,
 we view $\so_N$ as corresponding to the pair $(V_N,(\cdot,\cdot))$, where $(\cdot,\cdot)$ is represented by the symmetric
 matrix $J'=(J'_{ij})$ with $J'_{ij}=\delta_i^j, J'_{i,N+k}=J'_{N+k,i}=0, J'_{N+k,N+l}=\delta_{k+l}^{2m+2},\ \forall\ i,j\leq N,\ k,l\leq 2m+1$.}

\medskip
\noindent
 $\bullet$
  $e_m:=\sum_{j=1}^m E_{N+j,N+j+1}-\sum_{j=1}^m E_{N+m+j,N+m+j+1}$.

\medskip
  Recall the Lie algebra inclusion $\iota:\q\hookrightarrow U(\g,e)$ from~\cite[Section 1.6]{LT}, where $\q:=\zz_\g(e,h,f)$.
 For $(\g,e)=(\so_{N+2m+1},e_m)$ we have $\q\simeq \so_N$.
 We will also denote the corresponding centralizer of $e_m\in \so_{N+2m+1}$ and the Slodowy slice by $\zz_{N,m}$ and $\s_{N,m}$, respectively.

\begin{thm}\label{main 3}
   For $m\geq 1$, there is a unique isomorphism $\bar{\Theta}:H_m(\so_N,V_N)\iso \U(\so_{N+2m+1},e_m)$
 of filtered algebras such that $\bar{\Theta}\mid_{\so_N}=\iota\mid_{\so_N}$.
\end{thm}


\medskip
\noindent
 \emph{Sketch of the proof}.
 $\ $

  Notice that $\zz_{N,m}\simeq \so_N\oplus V_N\oplus \CC^m$ as vector spaces, where $\so_N\simeq \q=\zz_{N,m}(0), V_N\subset \zz_{N,m}(2m)$ and
 $\CC^m$ has a basis $\{\xi_0,\ldots,\xi_{m-1}\}$ with $\xi_i\in \zz_{N,m}(4m-4i-2)$.
  Here $\xi_{m-j}=e_m^{2j-1}\in \so_N$ for $1\leq j\leq m$, $V_N$ is embedded via $x_i \mapsto E_{i,N+2m+1}-E_{N+1,i}$,
 while $\so_N$ is embedded as a top-left $N\times N$ block of $\so_{N+2m+1}$.

  Let us recall that one of the key ingredients in the proof of~\cite[Theorem 7]{LT} was an additional $\ZZ$-grading $\Gr$ on the corresponding
 $W$-algebras.\footnote{\ Actually, as exhibited by the case of $\spn_{2n+2m}$, it suffices to
 have a $\ZZ_2$-grading.}
  In both cases of $(\ssl_{n+m},e_m),(\spn_{2n+2m},e_m)$ such a grading was induced from the weight-decomposition with respect to $\ad(\iota(h)),\ h\in \q$.

  If $N=2n$, same argument works for $\g=\so_{N+2m+1}$ as  well.
 Namely, consider $h\in \q\simeq \so_{2n}$ to be the diagonal matrix $I_n':=\diag(1,\ldots,1,-1,\ldots,-1)$.
 The operator $\ad(\iota(I_n'))$ acts on $\zz_{N,m}$ with zero eigenvalues on $\CC^m$, with even eigenvalues on $\so_N$, and with eigenvalues $\{\pm 1\}$ on $V_N$.

  However, there is no appropriate $h\in \q$ in the case of $N=2n+1$.
  Instead, such a grading originates from the adjoint action of the element
 $$\mathrm{g}_0:=(\underbrace{-1,\ldots,-1}_N,\underbrace{1,\ldots,1}_{2m+1})\in \mathrm{O}(N+2m+1).$$

  This element defines a $\ZZ_2$-grading on $U(\so_{N+2m+1})$ and further a $\ZZ_2$-grading $\Gr$ on the $W$-algebra $U(\so_{N+2m+1},e_m)$.
 The induced $\ZZ_2$-grading $\Gr'$ on $\gr U(\so_{N+2m+1},e_m)\simeq S(\zz_{N,m})$ satisfies the desired properties, that is,
 $\deg(\CC^m)=0,\ \deg(\so_N)=0,\ \deg(V_N)=1$.

  Therefore the algebra $U(\so_{N+2m+1},e_m)$ is equipped both with a Kazhdan filtration and a $\ZZ_2$-grading $\Gr$.
 Moreover, the corresponding isomorphism at the Poisson level is established in Theorem~\ref{main 4}.
 Now the proof proceeds along the same lines as in~\cite[Theorem 7]{LT}.  \ \ $\blacksquare$

\medskip
  Let us introduce some more notation:

\noindent
 $\bullet$
 Let $\bar{\iota}: \so_N\oplus V_N\oplus \CC^m\iso \zz_{N,m}$ denote the isomorphism from the proof of Theorem~\ref{main 3}.

\noindent
 $\bullet$
 Let $H^{\cl}_m(\so_N,V_N)$ be the Poisson counterpart of $H_m(\so_N,V_N)$ (compare to algebras $H^{\cl}_\z(\so_N,V_N)$).

\noindent
 $\bullet$
 Define $P_j\in \CC[\so_{N+2m+1}]$ by $\det(I_{N+2m+1}+tA)=\sum_{j=0}^{N+2m+1} P_j(A)t^j$.

\noindent
 $\bullet$
 Define $\{\bar{\Theta}_i\}_{i=0}^{m-1}\in S(\zz_{N,m})\simeq \CC[\s_{N,m}]$ by $\bar{\Theta}_i:=P_{2(m-i)}{_{\mid_{\s_{N,m}}}}$.

\medskip
 The following result can be considered as a Poisson version of Theorem~\ref{main 3}:

\begin{thm}\label{main 4}
 The formulas
    $$\bar{\Theta}^{\cl}(A)=\bar{\iota}(A),\ \bar{\Theta}^{\cl}(y)=\frac{(-1)^{\frac{m}{2}}}{2}\cdot\bar{\iota}(y),\ \bar{\Theta}^{\cl}(\z_k)=(-1)^{m-j}\bar{\Theta}_k$$
 define an isomorphism $\bar{\Theta}^{\cl}:H_m^{\cl}(\so_N,V_N)\iso S(\zz_{N,m})\simeq \CC[\s_{N,m}]$ of Poisson algebras.
\end{thm}

  The proof of this theorem proceeds along the same lines as for $\spn_{2N}$ (see~\cite[Theorem 10]{LT}).

\subsection{Consequences}
$\ $

 Let us now deduce a few results on the infinitesimal Hecke algebras of $(\so_N, V_N)$.

\begin{cor}
  Poisson varieties corresponding to arbitrary full central reductions of Poisson infinitesimal Hecke algebras $H_\z^{\cl}(\so_N, V_N)$
 have finitely many symplectic leaves.
\end{cor}


\begin{cor}

 (a) The center $Z(H_{\z}(\so_N,V_N))$ is a polynomial algebra in $\lfloor\frac{N+1}{2}\rfloor$
    generators.\footnote{\ Here we use the description of the center of the W-algebras, see~\cite[Theorem 5]{LT} for a reference.}

\noindent
 (b) The infinitesimal Hecke algebra $H_\z(\so_N, V_N)$ is free over its center $Z(H_{\z}(\so_N,V_N))$.

\noindent
 (c) Full central reductions of $\gr H_\z(\so_N, V_N)$ are normal, complete intersection integral domains.
\end{cor}


  Finally, the isomorphism of Theorem~\ref{main 3} provides
 the appropriate categories $\Oo$ for the algebras $H_m(\so_N,V_N)$
 (and hence for $H_\z(\so_N, V_N)$) once we have them for the finite $W$-algebras.
  The  categories $\mathcal{O}$ for the finite $W$-algebras were first
 introduced in~\cite{BGK} and were further studied in~\cite{L}.
 Namely, recall that we have an embedding $\mathfrak{q}\subset U(\g,e)$.
 Let $\mathfrak{t}$ be a Cartan subalgebra of $\mathfrak{q}$ and set
 $\g_0:=\mathfrak{z}_{\g}(\mathfrak{t})$. Pick an integral element
 $\theta\in \mathfrak{t}$ such that $\mathfrak{z}_{\g}(\theta)=\g_0$.
  By definition, the category $\mathcal{O}$ (for $\theta$) consists of
 all finitely generated $U(\g,e)$-modules $M$, where the action of
 $\mathfrak{t}$ is diagonalizable with finite dimensional eigenspaces
 and, moreover, the set of weights is bounded from above in the sense
 that there are complex numbers $\alpha_1,\ldots,\alpha_k$ such that
 for any weight $\lambda$ of $M$ there is $i$ with
 $\alpha_i-\langle\theta,\lambda\rangle\in \mathbb{Z}_{\leqslant 0}$.
  The category $\mathcal{O}$ has analogues of Verma modules,
 $\Delta(N^0)$. Here $N^0$ is an irreducible module over the
 $W$-algebra $U(\g_0,e)$, where $\g_0$ is the centralizer of
 $\mathfrak{t}$. In the case of interest $(\g,e)=(\so_{N+2m+1},e_m)$,
 we have $\g_0=\so_{2m+1}\times \mathbb{C}^N$ and $e$ is principal in $\g_0$.
  In this case, the $W$-algebra $U(\g_0,e)$ coincides with the center
 of $U(\g_0)$. Therefore $N^0$ is a one-dimensional space, and the
 set of all possible $N^0$ is identified, via the Harish-Chandra
 isomorphism, with the quotient $\h^*/W_0$, where $\h,W_0$ are a
 Cartan subalgebra and the Weyl group of $\g_0$ (we take the quotient
 with respect to the dot-action of $W_0$ on $\h^*$). As in the usual
 BGG category $\mathcal{O}$, each Verma module has a unique
 irreducible quotient, $L(N^0)$. Moreover, the map $N^0\mapsto
 L(N^0)$ is a bijection between the set of finite dimensional
 irreducible $U(\g_0,e)$-modules, $\h^*/W_0$, in our case, and the
 set of irreducible objects in $\mathcal{O}$. We remark that all
 finite dimensional irreducible modules lie in $\mathcal{O}$.



\section{Casimir element}

  In this section we determine the first nontrivial central element of the algebras $H_\z(\so_N,V_N)$.
 In the non-deformed case $\z=0$, we have $t_1:=(v,v)\in Z(H_0(\so_N,V_N))$. Similarly to Corollary~\ref{Casimir Poisson},
 this element can be deformed to a central element
 of $H_\z(\so_N,V_N)$ by adding an element of $Z(U(\so_N))$.

 In order to formulate the result, we introduce some more notation:

\noindent
 $\bullet$ Define $\omega_s:=\frac{\pi^{1/2}(s+N-1)!}{2^{s+N+1}}$ and $\mu_s:=\pi^{N-\frac{1}{2}}(s+1)!\omega_s^{-1},\ \nu_s:=-\frac{\mu_s}{s+1}$.

\noindent
 $\bullet$ For a sequence $\{\z_j\}_{j=0}^m$ define $\{a_j\}_{j=0}^m$ recursively via
  \  $\z_j=2\nu_{2j+1}\sum_{l=1}^{m+1-j}(-1)^{l+1}\binom{2j+2l}{2l-1}a_{j+l-1}$.

\noindent
 $\bullet$ Define a sequence of parameters $\{g_j\}_{j=1}^{m+1}$ via $g_j=2\mu_{2j-1}(-2a_{j-1}+\sum_{l=1}^{m+1-j} (-1)^{l+1}\binom{2j+2l}{2l}a_{j+l-1})$.

\noindent
 $\bullet$ Define a polynomial $g(z):=\sum_{j=1}^{m+1} g_jz^j$.

\noindent
 $\bullet$ Define $A(z)(x,y):=(x,A(1+z^2A^2)^{-1}y)\det(1+z^2A^2)^{-1/2}$ and $B(z):=\det(1+z^2A^2)^{-1/2}$.

\noindent
 $\bullet$ Let $[z^m]f(z)$ denote the coefficient of $z^m$ in the series $f(z)$.

\noindent
 $\bullet$ Define $C\in Z(U(\so_N))$ to be the symmetrization of $\Res_{z=0}g(z^{-2})\det(1+z^2A^2)^{-1/2}z^{-1}dz$.

\medskip
  Then we have:
\begin{thm}\label{main 5}
 The element $t_1':=t_1+C$ is a central element of $H_{\z}(\so_N,V_N)$.
\end{thm}

\begin{defn}
 We call $t_1'=t_1+C$ the \emph{Casimir element} of $H_\z(\so_N,V_N)$.
\end{defn}

\begin{rem}
 The same formula provides a central element of the algebra $H_m(\so_N,V_N)$, where $C\in Z(U(\so_N))[\z_0,\ldots,\z_{m-1}]$.
\end{rem}

  Theorem~\ref{main 5} can be used to establish explicitly the isomorphism $\bar{\Theta}$ of Theorem~\ref{main 3} in the same way as this has
 been achieved in~\cite[Section 4.6]{LT} for the $\gl_n$ case.

\medskip
\noindent
 \emph{Proof of Theorem~\ref{main 5}}.
 $\ $

  Commutativity of $t_1'$ with $\so_N$ follows from the following argument:
  $$[t_1,\so_N]=0\in H_0(\so_N,V_N)\Rightarrow [t_1,\so_N]=0\in H_\z(\so_N,V_N)\Rightarrow [t_1',\so_N]=0\in H_\z(\so_N,V_N).$$

 Let us now verify $[t_1+C,x]=0$ for any $x\in V_N$.

\noindent
 Identifying $U(\so_N)$ with $S(\so_N)$ via the symmetrization map and recalling (1), we get:
 $$[\sum x_i^2,x]=
   \sum_i x_i\int_{p\in S^{N-1}}\int_{q\in S^{N-2}(p)}(x_i,J_{p,q}x)\left(\int_{-\pi}^\pi 2c(\theta)\sin\theta e^{\theta J_{p,q}}d\theta\right)dqdp\ +$$
 $$\sum_i \int_{p\in S^{N-1}}\int_{q\in S^{N-2}(p)}\left(\int_{-\pi}^\pi 2c(\theta)\sin\theta e^{\theta J_{p,q}}d\theta\right)(x_i,J_{p,q}x)x_idqdp.$$
  Since $\sum_i x_i(x_i,J_{p,q}x)=J_{p,q}x$ and
 $v e^{\theta J_{p,q}}=e^{\theta J_{p,q}}(\cos \theta\cdot v-\sin\theta\cdot  J_{p,q}v)\ \mathrm{for}\ v\in V_N$, we have
\begin{equation}\label{Casimir1}
 [t_1,x]=\int_{p\in S^{N-1}}\int_{q\in S^{N-2}(p)}\int_{-\pi}^\pi 2c(\theta)\sin\theta e^{\theta J_{p,q}}(\sin\theta\cdot x+(1+\cos\theta)\cdot J_{p,q}x)d\theta dqdp.
\end{equation}
  The right hand side of~(\ref{Casimir1}) can be written as $[x,C']$, where
 $$C':=\int_{p\in S^{N-1}}\int_{q\in S^{N-2}(p)}\left( \int_{-\pi}^\pi c(\theta)(-2-2\cos\theta) e^{\theta J_{p,q}}d\theta\right)dqdp.$$
  Thus, it suffices to prove $C'=C$.

 The following has been established during the proof of Theorem~\ref{main 2}:
\begin{equation}\label{1st integral}
 \int_{p\in S^{N-1}}\int_{q\in S^{N-2}(p)}J_{p,q}^s dqdp=F_{s-1}=\mu_{s-1}[z^s]B(z),
\end{equation}
\begin{equation}\label{2nd integral}
 \int_{p\in S^{N-1}}\int_{q\in S^{N-2}(p)}(x,J_{p,q}y)J_{p,q}^s dqdp=I_{s;x,y}=\nu_s[z^{s-1}]A(z)(x,y).
\end{equation}

  Let $c(\theta)=c_0\delta_0+c_2\delta_0^{''}+c_4\delta_0^{(4)}+\ldots$ be the distribution from~(\ref{commutator}),
 where $\delta_0^{(k)}$ is the $k$-th derivative of the delta-function.
 Since
 $$\int_{-\pi}^\pi 2c(\theta)\sin\theta e^{\theta J_{p,q}}d\theta=
   2\sum_{j\geq 1}c_j\sum_{l=1}^{\lfloor \frac{j+1}{2} \rfloor}  (-1)^{l+1}\binom{j}{2l-1}J_{p,q}^{j-2l+1},$$
 formulas~(\ref{commutator}) and~(\ref{2nd integral}) imply
\begin{equation*}
   [x,y]=\Res_{z=0}\bar{\z}(z^{-2})A(z)(x,y)z^{-1}dz,
\end{equation*}
 where $\bar{\z}(z^{-2})=\sum_{j\geq 0}\bar{\z}_{j}z^{-2j}$ and
 $\bar{\z}_j=2\nu_{2j+1}\sum_{l\geq 1} (-1)^{l+1}\binom{2j+2l}{2l-1}c_{2j+2l}$.

 Comparing with $[x,y]=\Res_{z=0} \z(z^{-2})A(z)(x,y)z^{-1}dz$, we get $\bar{\z}(z^{-2})=\z(z^{-2})$ and so $c_{2s+2}=a_s$, where $a_{>m}:=0$.
  On the other hand,
 $$\int_{-\pi}^\pi c(\theta)(-2\cos\theta-2)e^{\theta J_{p,q}}d\theta=
   2\sum_{j\geq 0}c_j\left(-2J_{p,q}^j+\sum_{l=1}^{\lfloor j/2\rfloor}(-1)^{l+1}\binom{j}{2l}J_{p,q}^{j-2l}\right).$$
 Combining this equality with~(\ref{1st integral}), we find:
\begin{equation*}\label{parameters2}
  C'=\Res_{z=0}g(z^{-2})B(z)z^{-1}dz=C.
\end{equation*}
 This completes the proof of the theorem. \ \ \ \ \ $\blacksquare$


$\ $

\end{document}